\documentclass[12pt]{article}
\usepackage[T1]{fontenc}
\usepackage{multicol}
\usepackage{authblk}
\usepackage{mathtools}
\usepackage{amsmath,amsthm}
\usepackage{hyperref}
\hypersetup{
    colorlinks=true,
    citecolor=blue!90!black,
    pdftitle={Persistent Intrinsic Volumes},
    pdfpagemode=FullScreen,
    }
\usepackage{cleveref}
\usepackage{tikz}
\usepackage{bbold}
\usepackage{amssymb}
\usepackage{mathtools}
\usepackage{changepage} 
\usetikzlibrary{cd}

\DeclareMathOperator{\reach}{\mathrm{reach}}
\DeclareMathOperator{\rank}{\mathrm{rank}}
\DeclareMathOperator{\vect}{\mathrm{Vect}}
\DeclareMathOperator{\Id}{\mathrm{Id}}
\DeclareMathOperator{\clarke}{\partial^* \!}
\DeclareMathOperator{\Conv}{\mathrm{Conv}}

\DeclareMathOperator{\Nor}{\mathrm{Nor}}

\DeclareMathOperator{\Tan}{\mathrm{Tan}}
\DeclareMathOperator{\dgm}{\mathrm{dgm}}
\DeclareMathOperator{\im}{\mathrm{im}}
\DeclareMathOperator{\one}{\mathbb{1}}
\DeclareMathOperator{\Vol}{\mathrm{Vol}}

\newcommand{\Sphere}{\mathbb{S}}
\DeclareMathOperator{\diff}{\mathrm{d} \!}
\DeclareMathOperator{\tq}{ \, | \,  }

\newcommand{\R}{\mathbb{R}}
\newcommand{\N}{\mathbb{N}}
\newcommand{\module}[1]{\left\lvert #1 \right\rvert}
\newcommand{\norme}[1]{\left\lvert \left \lvert #1 \right \rvert \right \rvert}
\newcommand{\eps}{\varepsilon}
\newcommand{\complementaire}[1]{^\neg #1}
\newcommand{\scal}[2]{ \langle #1 , #2 \rangle}
\newtheorem{theorem}{Theorem}[section]
\newtheorem*{theorem*}{Theorem}
\newtheorem{corollary}[theorem]{Corollary}
\newtheorem{lemma}[theorem]{Lemma} 
\newtheorem{definition}[theorem]{Definition}
\newtheorem{proposition}[theorem]{Proposition}
\newtheorem{remark}[theorem]{Remark}

\author[1]{David Cohen-Steiner}
\author[1,2]{Antoine Commaret\thanks{Corresponding author.}}
\affil[1]{\textit{Centre INRIA d'Université Côte d'Azur, Valbonne, France}}
\affil[2]{\textit{Laboratoire Jean-Alexandre Dieudonné, Nice, France}}

\title{Persistent Intrinsic Volumes}
\date{August 2026}

\begin{document}

\maketitle

\begin{abstract}
We develop a new method to estimate the area, and more generally the intrinsic volumes, of a compact subset $X$ of $\R^d$ from a set $Y$ that is close in the Hausdorff distance. This estimator enjoys a linear rate of convergence as a function of the Hausdorff distance under mild regularity conditions on $X$. Our approach combines tools from both geometric measure theory and persistent homology, extending the noise filtering properties of persistent homology from the realm of topology to geometry. Along the way, we obtain a stability result for intrinsic volumes.
\end{abstract}



\emergencystretch 3em
\section{Introduction}

Geometric inference deals with the retrieval of geometric characteristics of an unknown shape $X$ from a finite sample of $X$. Among such geometric characteristics, the \emph{intrinsic volumes} $V_0(X), \dots, V_d(X)$ of a subset $X$ of a Euclidean space $\R^d$ are of particular importance. 
Indeed, these quantities enjoy general properties such as additivity, isometry invariance and positive homogeneity, making them relevant in various fields such as integral geometry, where intrinsic volumes are known as \emph{quermassintegrals}, or in differential geometry
where they are called \emph{Lipschitz-Killing curvatures}.
 Notably, the codimension 1 intrinsic volume $V_{d-1}(X)$ is a multiple of the boundary area of $X$, whose estimation is an important question in geometric inference. It is of practical relevance in a large swath of scientific and industrial domains. For example, biologists seek to measure the area of tissues such as lungs \cite{LungSurface}, bones \cite{NecrosisSurface}, blood vessels \cite{VascularSurface} or brains \cite{BrainSurface}. Applications in other fields such as chemistry, geology, engineering are so numerous that they are difficult to summarize.
 From a mathematical perspective, we want to understand under which conditions and at which rate we are able to retrieve the intrinsic volumes of $X$ from an approximating set $Y$. In order to achieve this, we introduce a novel method to estimate the intrinsic volumes of $X$ using techniques from persistent homology and geometric measure theory. This allows to weaken the regularity conditions required in the current literature. Furthermore, we obtain rates of convergence with respect to the Hausdorff distance $d_H(X,Y)$, allowing $X$ to be approximated by point clouds or voxels.

Intrinsic volumes were first defined for convex subsets of $\R^d$  \cite{steiner, konvexen_korper} and sets whose boundary is a smooth submanifold of $\R^d$ \cite{weyl}. In either case, while
$V_d(X)$
coincides with $\mathcal{H}^d(X)$, the $d$-dimensional Hausdorff measure of $X$ (i.e. the volume of $X$), the remaining intrinsic volumes $V_i(X)_{0 \leq i \leq d-1}$ depend on the curvature of its boundary. If $X$ is an open set of $\R^d$ bounded by a hypersurface, denote by $(\kappa_i(x))_{1 \leq i \leq d-1}$ its principal curvatures at $x \in \partial X$. Then, for all $0 \leq i \leq d-1$, we have (see for instance \cite[3.1]{ZahleRataj}):
\[ V_i(X) = \frac{1}{(d-i) \omega_{d-i}} \int_{\partial X} \Sigma_{d-i-1}(\kappa_1, \dots, \kappa_{d-1}) d \omega (x) \]
where $\Sigma_i(X_1, \dots, X_{d-1}) = \sum_{1 \leq j_1 < \dots < j_i \leq d-1} X_{j_1} \dots X_{j_i} $ is the elementary symmetric polynomial in $d-1$ variables of degree $i$, $\omega_i$ the $i$-th volume of the unit ball in dimension $i$ and $d \omega$ the canonical volume form of $\partial X$. In particular, up to multiplicative constants, $V_{d-1}(X)$ is the $(d-1)$-dimensional Hausdorff measure $\mathcal{H}^{d-1}(\partial X)$ i.e. the area of $\partial X$, $V_{d-2}(X)$ is the integral of the mean curvature and $V_0(X)$ is the integral of the Gauss curvature which by the Gauss-Bonnet theorem is the Euler characteristic $\chi(X)$ of $X$. Similar formulas exist when $X$ is a hypersurface of dimension $k$, in which case $V_k(X)$ is a multiple of its area.   
When $X$ is a convex polytope, the $i$-th intrinsic volume is the sum of the volumes of its $i$-dimensional facets multiplied by their appropriately normalized exterior cone angle.

The problem of extending the definition of intrinsic volumes to broader classes of sets was first posed by Milnor \cite{steiner_volumes}. Since then, consistent definitions were given for Whitney-stratified sets \cite{IntegralGeometryTameSets}, sets in a o-minimal structure \cite{Bernig}, and general unions of sets with positive reach \cite{ZahleUnionReach} using respectively stratified Morse theory, o-minimal axioms and geometric measure theory. Fu \cite{FuSubAnalytic} gave general conditions for sets to have well-defined intrinsic volumes and showed that those conditions hold for subanalytic sets. Most of these definitions employ the theory of \textit{normal cycles} developed in \cite{ZahleCurrent}, which uses the theory of currents \cite{GMT}. In our framework we give a definition of intrinsic volumes for offsets of sets with positive $\mu$-reach without using the language of currents, the $\mu$-reach being a relaxed version of the reach of Federer \cite{FedererCurvature} (see \Cref{def:mu_reach}).

\paragraph{Previous work} Previous results focused mostly on the estimation of the area 
$V_{d-1}(X)$. One approach is to estimate the area by the area of a piecewise linear reconstruction of the data. For example, the tangent complex triangulation \cite{TangentComplex} guarantees that this estimator converges to $V_{d-1}(X)$ at a linear rate in the Hausdorff distance $d_H(X,Y)$ when $X$ is a smooth submanifold of $\R^d$ and $Y$ a noise-free sample.
Using Crofton's formula, \cite{Aaron} builds an estimator for the surface area $V_{d-1}(X)$ from a point cloud sample $Y$ and obtains a square-root rate of convergence
$O(d_H(X,Y)^{1/2})$ in the general case. Other works have focused on the retrieval of \textit{curvature measures}, which are local, more informative versions of the intrinsic volumes. A convergence rate of $d_H(X,Y)^{1/2}$ for the curvature measures was obtained in \cite{BoundaryMeasures} under the condition $d_H(X,Y) \leq \reach(X)$. 
Chazal et al. \cite{ChazalStability} obtain the convergence of the curvature measures of $Y^r$ to that of $X^r$ at a square root rate, for any fixed $r > 0$, where $Z^{r} = \{ x \in \R^d \tq d_Z(x) \leq r\}$ is the $r$-offset of any subset $Z$ of $\R^d$. 

Linear convergence rates were obtained for the estimation of the first intrinsic volume using persistent homology of height functions and Crofton's formula.
Authors in \cite{CurveInequalities} showed a linear rate of convergence of $V_{1}(Y)$ to $V_{1}(X)$ with respect to the Fréchet distance between $X$ and $Y$ when they are both compact surfaces of $\R^3$ or both curves in $\R^d$ assuming their total absolute curvature is bounded.
Building on these ideas, Edelsbrunner et al. \cite{IntrinsicEdelsbrunner} obtained more recently a linear rate of convergence for the first intrinsic volume of voxelizations of smooth sets, in addition to showing the convergence of all intrinsic volumes of spheres voxelized with ever increasing precision. 

Another line of research has focused on non-deterministic geometric inference from uniform samples of convex sets. Notably, the authors of \cite{Convex1, Convex2} worked with convex sets with a $C^{k}$ boundary, where $k \geq 2$. The expected intrinsic volumes of the convex hull of the sample were shown to converge each to the ideal intrinsic volume at a rate of $C_X n^{-2/(d+1)}$ where $n$ is the number of sample points and $C_X$ a constant depending on $X$.  It was also proven that this result does not hold with a mere $C^1$ boundary condition, suggesting that finding an estimator that is robust to the lack of regularity is difficult.
 
\paragraph{Contributions} In this paper, we define quantities $V^{\eps}_i(Y)$ depending on a parameter $\eps$, that approach $V_i(X)$ at a linear rate in $d_H(X,Y)$ assuming only mild regularity conditions on $X$. This rate is easily seen to be optimal. To the best of our knowledge, these are the first estimators that come with theoretical guarantees beyond sets with positive reach. Even for the basic problem of estimating the boundary area of 3-dimensional object with reach zero, we are not aware of any other provably correct method. 
 
 \begin{theorem*}[Main Results]
 Let $X, Y$ be two compact sets of $\R^d$ and let $\mu \in (0,1], \eps > 0$ be such that $d_H(X,Y) \leq \eps \leq \frac{1}{4} \reach_{\mu}(X)$.
 Then we have:
 \begin{equation*}
\module{V^{\eps}_i(Y) - V_i(X^{2\eps})} \leq C_d K(X^{2\eps}) \frac{\eps}{\mu}, 
 \end{equation*}
where $C_d$ is a constant depending on $d$, $K(X^{2\eps}) \coloneqq \mathcal{H}^d(X^{2\eps}) + \mathcal{H}^{d-1}(\Nor(X^{2\eps} ))$ and $\Nor(X^{2\eps})$ is the unit normal bundle of $X^{2\eps}$.

 Further assuming that $\reach(X) > 0$, we prove a linear rate of convergence for the intrinsic volumes of offsets:
\begin{equation*}
\module{V_i(X) - V_i(X^{2\eps})} \leq C_d \left( K(X) + K(X^{2\eps}) \right )\frac{\eps}{\mu}.
\end{equation*}
 \end{theorem*}
It is worth noting that the second claim holds even when $\reach(X)$ is arbitrarily smaller than $\eps$, a case for which, to the best of our knowledge, no quantitative convergence result between the curvatures of $X^{2\eps}$ and $X$ was known. We also conjecture that this claim holds when $X$ is subanalytic. Taken together, the two claims above provide a way to estimate the intrinsic volumes of an unknown shape $X$ with positive reach from a Hausdorff approximation $Y.$
From this point of view, the condition that $\reach(X)$ is positive is not restrictive since sets with positive reach form a dense family of compact subsets for the Hausdorff distance.

A byproduct of our methods is an answer to the second open question asked by Milnor in \cite{steiner_volumes}: In which sense do $X$ and $Y$ have to be close to guarantee that their intrinsic volumes are close? It turns out that the existence of a $C^0$-controlled homotopy equivalence (see \cite{controlled_homotopy} for a related notion) is sufficient, assuming a bound on the volume of the unit normal bundle of both sets. 
More precisely, say that 
two subsets $X$ and $Y$ of a Euclidean space
 are \emph{$(\eps, \delta)$-homotopy equivalent} if there exist two continuous maps $f : X \to Y$, $g : Y \to X$ such that ${\norme{f - \Id}_{\infty} \leq \eps}$, $\norme{g - \Id}_{\infty} \leq \eps$ and such that there exist homotopies $H_1$ (resp. $H_2$) between $f \circ g$ and $\Id_Y$ (resp.  $g \circ f$ and $\Id_X$) satisfying $\norme{H_1(t,\cdot) - \Id_Y}_{\infty} \leq 2\delta$ (resp. $\norme{H_2(t,\cdot) - \Id_X}_{\infty} \leq 2\delta$) for all $t \in [0,1]$.

\begin{theorem}[]
\label{thm:milnor_2}
Let $X$ and $Y$ be two compact subsets of $\R^d$ with positive reach. If $X$ and $Y$ are $(\eps, \delta)$-homotopy equivalent for $\eps$ and $\delta$ positive, we have:
\begin{equation*}
    \module{V_i(X) - V_i(Y)} \leq C_d \max(\eps, \delta)\left ( K(X) + K(Y) \right ).
\end{equation*}
\end{theorem}

For our inference problem, a naive approach would be to estimate the intrinsic volumes of $X$ by the intrinsic volumes of small offsets of $Y.$ However, this leads to a trade-off between the bias induced by too large offset parameters and the spurious geometric details that come with small offset parameters. Optimizing this trade-off yields a sublinear rate of convergence. We use the noise-filtering properties of persistent homology to improve over this sublinear behavior by studying the inclusion $Y^{\eps} \subset Y^{3\eps}$ instead, similar to the usual method for estimating Betti numbers \cite{ChazalWeakFeatureSize, PersistenceStability}. Our approach uses the principal kinematic formula from integral geometry to express the intrinsic volumes as integrals of certain Euler characteristics. We then define our persistent intrinsic volumes by replacing these Euler characteristics with persistent Euler characteristics associated with the pair $Y^{\eps} \subset Y^{3\eps}$.  A stability theorem for image persistence then allows us to prove a linear rate of convergence for our estimators.

\section*{Outline}
\begin{itemize}
    \item In \Cref{sec:distance}, we describe some basic properties of distance functions and we introduce our regularity conditions.
    \item In \Cref{sec:intrinsic_volumes}, we define the intrinsic volumes, and the class of sets \textit{admitting a normal bundle}. Geometric measure theory provides the \textit{principal kinematic formula} which serves as the link between the topological and geometric properties of a compact set admitting a normal bundle.
    \item In \Cref{sec:persistance} and \Cref{sec:image_persistence}, we give some background on persistent homology and prove a stability theorem for image persistence modules associated with sublevel set filtrations of not necessarily smooth functions. We also prove a bound on the number of bars in an image persistence diagram.
    \item In \Cref{sec:persistent_volumes}, we state and prove our main results.
    \item In \Cref{sec:MonteCarlo}, we discuss the computational tractability of persistent intrinsic volumes and we give a Monte-Carlo based algorithm to approximate them.
\end{itemize}

\section{Distance functions and $\mu$-reach of compact sets}
\label{sec:distance}
In this section, we give definitions pertaining to non-smooth analysis and distance functions \cite{Clarke1975GeneralizedGA}. 

\begin{definition}
Let $X \subset \R^d$ be a closed set. 
\begin{itemize}
    \item The distance to $X$ is the map $d_X : x \in \R^d  \mapsto  \inf_{z \in X} \norme{z-x}$. When $X= \{x\}$ is a singleton, we write $d_X = d_x$ to ease notations. We let $\Delta(X) \coloneqq d_X(0)$.

    \item The set of \emph{closest points to $x$ in $X$} is denoted by $\Gamma_X(x)$.
    \item The \emph{Clarke gradient} of $d_X$ at any point $x \notin X$ is the convex hull of directions opposite to the closest points to $x$, see \Cref{fig:clarke}. 
    \begin{equation*}
\clarke d_X(x) \coloneqq \Conv \left ( \left \{ \frac{x-z}{\norme{x-z}} \, \Bigr | \, z \in \Gamma_X(x)  \right \} \right ).    
\end{equation*}
\end{itemize} 
\end{definition}

\begin{figure}[h!]
    \centering
    \includegraphics[scale = 0.35]{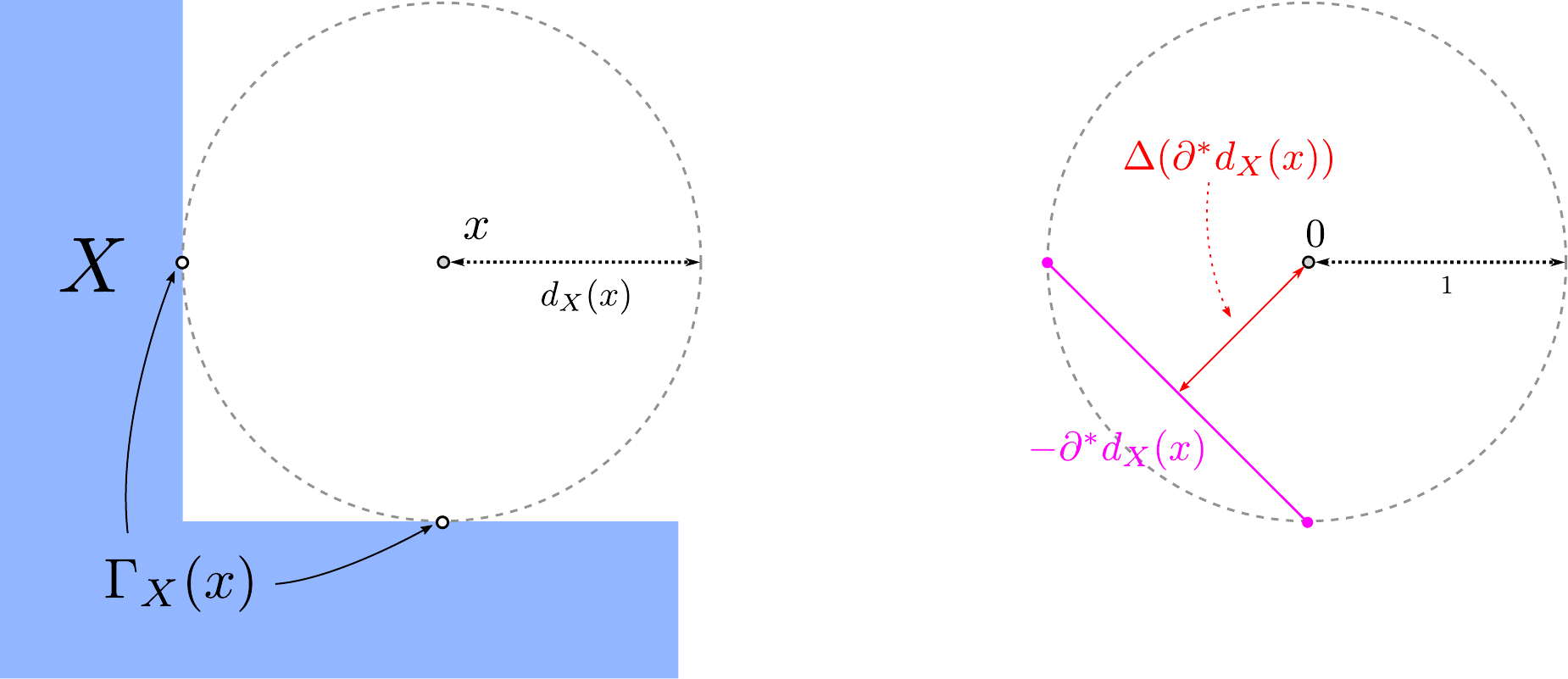}
    \caption{Clarke gradient of $d_X$ (right) as a function of the closest points to $x$ in $X$ (left).}
    \label{fig:clarke}
\end{figure}

The quantity $\Delta \! \left (\clarke d_X(x) \right )$ coincides with the norm of the generalized gradient $\norme{\nabla d_X(x)}$ defined by Lieutier in \cite{Lieutier}, which is classically used in the definition of the $\mu$-reach of a compact set. 

\begin{definition}[Reach and $\mu$-reach of a subset of $\R^d$]
\label{def:mu_reach}
Let $X \subset \R^d$ and $\mu \in (0,1]$. We define the $\reach$ and the $\mu$-$\reach$ of $X$ by:

\begin{itemize}
\item $\reach(X) \coloneqq \sup \left \{t > 0 \, | \, 0 < d_X(x) < t \implies \mathrm{card} \; \Gamma_X(x) = 1 \right \}$,
\item $\reach_{\mu}(X) \coloneqq \sup \left \{t > 0 \, | \, 0 < d_X(x) < t \implies \Delta (\clarke d_X(x)) \geq \mu \right \}$.
\end{itemize}
The function $\mu \mapsto \reach_{\mu}(X)$ is non-increasing, and $\reach_{1}(X) = \reach(X)$.
\end{definition}

We need the following theorem \cite{MorseTubular} for the specific setting where $\phi$ equals the distance function to $X$. With this choice the theorem yields that for any compact set $X$ with $\reach_{\mu}(X) > 0$, a sufficiently small offset of $X$ deformation retracts to $X$ itself.

\begin{theorem}[Approximate inverse flow of Lipschitz functions]
\label{thm:flow}
   Let $\phi: \R^d \to \R$ be a Lipschitz real-valued function and let $a < b \in \R$. Assume that: \[ \inf  \{ \Delta( \clarke \phi(x)) \tq x \in \phi^{-1}(a,b] \} = \mu > 0. \] Then for every $\sigma  \in (0, \mu)$, there exists a continuous function
 \[C^{\sigma}_{\phi}: \left \{
 \begin{array}{ccc}
      [0,1] \times \phi^{-1}( \infty, b]& \to & \phi^{-1}(-\infty, b]  \\
      (t,x) & \mapsto & C^{\sigma}_{\phi}(t,x) 
 \end{array} \right.
 \]
 such that:
 \begin{itemize}
     \item For any $s > t$ and $x$ such that $C^{\sigma}_{\phi}(s,x) \in \phi^{-1}(a,b]$, we have:
 \begin{equation*}
\phi(C^{\sigma}_{\phi}(s,x)) - \phi(C^{\sigma}_{\phi}(t,x)) \leq -(s-t)(b-a);
 \end{equation*} 
 \item  For any $t \in [0,1]$, $x \in \phi^{-1}(\infty, a]$ implies $C^{\sigma}_{\phi}(t,x) = x$;
 \item For any $x \in \phi^{-1}( -\infty,b]$, the map $s \mapsto C^{\sigma}_{\phi}(s,x)$ is $\frac{b-a}{\mu - \sigma}$-Lipschitz.
 \end{itemize}
 
 In particular, $C^{\sigma}_{\phi}$ is a deformation retraction between $\phi^{-1}(-\infty, a]$ and $\phi^{-1}(-\infty,b]$.
\end{theorem}

\section{Intrinsic Volumes and normal bundles}
\label{sec:intrinsic_volumes}

\subsection*{Background on Intrinsic Volumes}
 Intrinsic volumes of subsets of $\R^d$ are geometric quantities of importance in geometric measure theory. They were first defined by Steiner \cite{steiner} for compact convex sets and by Weyl \cite{weyl} for smooth submanifolds. Federer \cite{FedererCurvature} extended their definition to the class of sets with \emph{positive reach}. Specifically, he proved the \emph{tube formula} which states that the volume of small offsets of a set with positive reach is a polynomial whose coefficients provide a definition for the intrinsic volumes, generalizing the definitions of Steiner and Weyl: 
\[ \forall \; 0 \leq r \leq \reach(X), \; \mathcal{H}^{d}(X^r) \eqqcolon \sum_{i=0}^{d} \omega_i V_{d-i}(X) r^i, \]
where for any $s \geq 0$, $\mathcal{H}^s$ denotes  the $s$-Hausdorff measure of $\R^d$.
The polynomial in the right-hand side is called \emph{the Steiner polynomial of X} and we denote it by $Q_X$.

\begin{proposition}[Properties of intrinsic volumes]
Let $X,Y \subset \R^d$ be two sets of positive reach. The intrinsic volumes $V_i$ satisfy the following properties for all $0 \leq i \leq d$:
\begin{description} 
    \item[(Isometry Invariance)] For any isometry $g : \R^d \to \R^d$, $X$ and $g(X)$ have the same intrinsic volumes.
    \item[(Hausdorff Continuity)] The restriction of $V_i$ to the class of convex subsets of $\R^d$ is continuous in the Hausdorff metric;
    \item[(Additivity)] If both $X \cup Y$ and $X \cap Y$ have positive reach, we have:
    \[ V_i(X \cup Y) + V_i(X \cap Y) = V_i(X) + V_i(Y).\]
\end{description}
\end{proposition}

The celebrated Hadwiger theorem \cite{hadwiger,klain1997introduction} states that intrinsic volumes form a basis of the functionals satisfying these three properties over the class of compact convex subsets of $\R^d$. The additivity property allows in particular to extend the definition of intrinsic volumes to e.g., finite unions of convex sets, using the inclusion-exclusion principle.\\

The principal kinematic formula (see for instance \cite[Part 6]{FedererCurvature} for the paper of Federer or \cite[Chapter 6]{ZahleRataj} for a more modern presentation) connects the intrinsic volumes of sets of positive reach $A,B$ with the integral of $V_i(A \cap g(B))$ with respect to the Haar measure over the group of affine isometries $g$. In particular, since from the Gauss-Bonnet theorem \cite[4.8]{ZahleRataj} $V_0$ is the Euler characteristic, taking $i = 0$ and $B$ a ball of radius $r > 0$ this formula reads as an extension of the tube formula when the offset parameter goes beyond the reach.

\begin{theorem}[Particular case of the Principal Kinematic Formula]
\label{th:kinematic_reach}
Let $X \subset \R^d$ be a compact set with positive reach. Then for any $r > 0$, 
\begin{equation*}
\int_{\mathbb{R}^d} \chi(X \cap B(x,r)) \mathrm{d}x = Q_X(r). 
\end{equation*}
\end{theorem}
 
 Intrinsic volumes have been extended to other classes of sets while preserving the properties and results listed above, notably the class of subanalytic sets \cite{FuSubAnalytic}, the $U_{PR}$ class i.e., generic unions of sets with positive reach \cite{ZahleUnionReach} and the WDC class \cite{wdc}. We will use another extension of intrinsic volumes, namely to the class of sets admitting a normal bundle, which we define in the following section.

\subsection*{Sets admitting a normal bundle}
We introduce a class of sets defined so as to contain small tubular neighborhoods of sets with positive $\mu$-reach for any $\mu \in (0,1]$. 
We recall \cite{GMT} that the \emph{tangent cone} $\Tan(X,x)$  of a set $X \subset \R^d$ at a point $x \in X$ is defined as the cone spanned by the set of limits of sequences of the form $(x_n-x)/\norme{x_n -x}$ where $x_n$ is a sequence in $X$ converging to $x$ and not equal to $x$ for all $n \in \N$.  When $X$ has positive reach, the \emph{normal cone} of $X$ at $x$ is 
$ \Nor(X,x) \coloneqq \Tan(X,x)^{\text{o}}$
    where for any set $S \subset \R^d$, its dual cone is 
    $S^{\text{o}} \coloneqq \{ u \in \R^d \tq  \scal{u}{v} \leq 0 , \,  \forall \, v \in S \}.$
    
    \begin{definition}[Sets admitting a normal bundle] Let $X \subset \R^d$.
When $\complementaire{X} \coloneqq \overline{\R^d \setminus X}$ has positive reach and is a Lipschitz domain\footnote{A Lipschitz domain is a subset of $\R^d$ that is locally the epigraph of a Lipschitz function.}, define
    \[ \Nor(X,x) \coloneqq - \Nor(\complementaire{X},x). \]
    This definition is consistent in case both $X, \complementaire{X}$ have positive reach. If either $\reach(X) > 0$ or both $\reach(\complementaire{X}) > 0$ and $X$ is a Lipschitz domain, we say that \emph{$X$ admits a normal bundle $\Nor(X)$} with
    \[\Nor(X) \coloneqq \bigcup_{x \in \partial X} \{ x \} \times (\Nor(X,x) \cap \Sphere^{d-1}). \]
\end{definition}

The pair $(x,n) \in \Nor(X)$ is said to be \emph{regular} when $\Tan(\Nor(X), (x,n))$ is a {$(d-1)$-dimensional} vector space. 
The structure of the tangent spaces of $\Nor(X)$ at regular pairs defines pointwise principal curvatures. They are essentially a generalization of the classical principal curvatures from submanifolds of $\R^d$ to sets admitting a normal bundle, allowing to consider integrals of curvatures (see \Cref{def:Rmass}) for non-smooth sets.

\begin{proposition}[Tangent spaces of normal bundles \cite{RatajLegendrian}]
\label{prop:structure}
Let $X$ be a compact subset of $\R^d$ admitting a normal bundle $\Nor(X)$. Then pairs in $\Nor(X)$ are $\mathcal{H}^{d-1}$-almost all regular, and for any regular pair $(x,n) \in \Nor(X)$, there exist:
\begin{itemize}
    \item A family
$\kappa_1, \dots , \kappa_{d-1}$ in $\R \cup \{ \infty \}$ called \emph{principal curvatures of $X$ at $(x,n)$};
    \item A family $b_1, \dots, b_{d-1} \in \R^d$ of vectors orthogonal to $n$ called \emph{principal directions} at $(x,n)$,
such that the family $\left ( \frac{1}{\sqrt{1 + \kappa^{2}_i}}b_i , \frac{\kappa_i}{\sqrt{1+\kappa_i^2}}b_i \right )_{1 \leq i \leq d-1}$ forms an orthonormal basis of $\Tan(\Nor(X), (x,n))$.
    \end{itemize}
\end{proposition}

Principal curvatures $\kappa_i = \kappa_i(x,n)$ are unique up to permutations.
When $\reach(X) > 0$, every principal curvature $\kappa_i$ is larger than $-\reach(X)^{-1}$. Moreover, assume that $\partial X$ is a $C^{1,1}$ hypersurface, i.e., there exists a locally Lipschitz Gauss map $x \mapsto n(x)$ of outward pointing normals. Then $n$ is differentiable $\mathcal{H}^{d-1}$-almost everywhere and at every point $x$ where it is differentiable the multiset of principal curvatures at $(x,n(x))$ coincides with the multiset of eigenvalues of the differential of $n$, i.e., the classical principal curvatures. 
At $(x,n)$ where $x$ is a non-smooth point of $\partial X$ and $n$ belongs to $\Nor(X,x)$, a set of positive reach might have infinite principal curvatures. For example, if $X$ is a polyhedron, principal curvatures across a face of $\partial X$ of codimension at least one are infinite.

We are now able to define the $R$-curvature mass of a set admitting a normal bundle. This quantity will appear in the explicit convergence bounds of our main results.
\begin{definition}[$R$-curvature mass of a compact set admitting a normal bundle]
\label{def:Rmass}
Let $X \subset \R^d$ be a compact subset of $\R^d$ admitting a normal bundle and let $R >0$. The $R$-curvature mass of $X$ is defined by: 
\begin{equation*}
M_R(X) \coloneqq \int_{0}^R \int_{\Nor(X)} \prod_{i=1}^{d-1} \frac{\module{1 + t \kappa_i(x,n)}}{\sqrt{1 + \kappa_i(x,n)^2}} \diff \mathcal{H}^{d-1}(x,n) \diff t.     
\end{equation*}

\end{definition}

\begin{proposition}[R-curvature mass of a set with smooth boundary]
When $X$ has a $C^{1,1}$ $(d-1)$-hypersurface as boundary, this quantity can be expressed as an integral over $\partial X$ with its volume form $d \omega_{\partial X}$ inherited from the canonical volume form of $\R^d$:
\begin{equation*}
  M_R(X) = \int_{0}^R \int_{\partial X} \prod_{i=1}^{d-1} \module{1 + t \kappa_i}  d \omega_{\partial X}\diff t.   
\end{equation*}
\end{proposition}

\begin{proof}
If $\partial X$ is smooth (or at least $C^{1,1}$) we obtain the equality with the definition by applying the change of variable with the bijective map $(x,n(x)) \mapsto x$ whose Jacobian is $\prod_{i=1}^{d-1} \frac{1}{\sqrt{1 + \kappa^2_i(x)}}$.
\end{proof}

\begin{proposition}[Bounds on $M_R(X)$]
\label{prop:bound}
Let $X$ admit a normal bundle. \\
If $R \leq \reach(X)$, we have:
\begin{equation*}
M_R(X) = \sum_{i=1}^{d} R^{i}\omega_i V_{d-i}(X).
\end{equation*}
For any $R > 0$, we have:
\begin{equation*}
M_R(X) \leq C_{R}(d) \mathcal{H}^{d-1}(\Nor(X)),
\end{equation*}
where 
\[ C_R(d) \coloneqq \int_{0}^R (1 +t^2)^{\frac{d-1}{2}} \diff t. \]
\end{proposition}

\begin{proof}The equality comes from the fact that for any $0 \leq t < \reach(X)$, $1 + t \kappa_i \geq 0$ and thus the coefficient in front of $t^{i-1}$ integrates over $\Nor(X)$ to $i V_{d-i}(X)\omega_i$ for any $1 \leq  i \leq d$. The inequality comes from the elementary fact that ${\frac{\module{1 + t x}}{\sqrt{1 + x^2}} \leq \sqrt{1 +t^2}}$ for any $t, x \in \R$.
\end{proof}

When $X$ is a Lipschitz domain and admits a normal bundle, the work of Rataj \& Zähle \cite[Theorem 9.35]{ZahleRataj} shows that intrinsic volumes are still defined and enjoy the principal kinematic formula. This leads to an extension of \Cref{th:kinematic_reach}:
\begin{proposition}[Definition of the Steiner polynomial]
Let $X$ be a set admitting a normal bundle.
The quantity \begin{equation*} Q_X(r) = 
    \int_{\R^d} \chi(X \cap B(x,r)) \diff x
\end{equation*} is a polynomial of degree $d$ in $r$  which we call the Steiner polynomial of $X$. Its coefficients are the intrinsic volumes of $X$ up to the same multiplicative constants as in the positive reach case.
\end{proposition}

Offsets of possibly irregular sets often admit a normal bundle.
Indeed, for every $0 < t < \reach_{\mu}(X)$ we know that 
$\complementaire{(X^t)}$ is a Lipschitz domain with positive reach \cite{ChazalDoubleOffsets}, showing that $X^t$ admits a normal bundle. When $X$ is a subanalytic subset of $\R^d$ all but a finite number of offsets $X^{t}$ have a positive $\mu$-reach for some $\mu \in (0,1]$, see \cite{FuSubAnalytic}. As a consequence, all but a finite number of offsets of a subanalytic set admit a normal bundle.

\section{Basics in persistent homology}
\label{sec:persistance}
This section gives a summary of the classical notions of persistence theory so as to make the paper more self-contained. Further information about this topic can be found in \cite{StructureStability}.

\begin{definition}[Persistence modules]
Let $\mathbb{K}$ be a field. A persistence module $M$ is a functor $\R \to \vect_{\mathbb{K}}$, that is a collection of vector spaces $(M_t)_{t \in \R}$ and maps $\phi^{t}_s : M_s \to M_t$ for any $s \leq t \in \R$, such that $\phi^u_{t} \circ \phi^{t}_s = \phi^{u}_s$ for any ordered triple $s \leq t \leq u$ in $\R$.
\end{definition}

We say that a persistence module is \emph{pointwise finite-dimensional} when $\dim(M_t) < \infty$ for all $t \in \R$. Such modules admit the following structure theorem:

\begin{proposition}[Interval decomposition of persistence modules]

Given an interval $I \subset \R$, $\one_{I}$ is the persistence module defined by
\[ 
(\one_{I})_t \coloneqq \left \{
\begin{array}{cc}
     \mathbb{K} & \text{ when } t \in I  \\
     0 & \text{ otherwise} 
\end{array} \right. \]
and such that the linear maps between $(\one_{I})_s \to (\one_{I})_{t} $ are the identity when $s < t \in I$.
 Let $M$ be a pointwise finite-dimensional persistence module. Then there exists a multiset $\mathcal{I}$ of intervals of $\R$ such that:
\[ M = \bigoplus_{I \in \mathcal{I}} \one_{I}. \]

 The \emph{persistence diagram} $\dgm(M)$ associated to the persistence module $M$ is the multiset of points in $\R^2$ whose coordinates are the bounds of the intervals decomposing $M$\!:
\[\dgm(M) \coloneqq \bigsqcup_{I \in \mathcal{I}} \{ (\inf I, \sup I) \}.\]
Elements of a persistence diagram are usually referred to as either intervals or bars.
\end{proposition}

There exists a natural distance over the space of persistence modules:

\begin{definition}[Morphisms, $\delta$-interleavings and interleaving distance]
Let $M,N$ be two persistence modules and let $\delta \geq 0$.\\
\begin{itemize}
    \item $M^{\delta}$ denotes the persistence module $(M_{t + \delta})_{t \in \R}$, i.e., $M$ shifted by $\delta$.  
    \item A morphism $j$ between two persistence modules $N$ and $M$ is a natural transformation between functors, i.e., a collection of linear maps $(j_t)_{t \in \R} : N_t \to M_t$ such that the left diagram in \Cref{fig:commute} commutes.
    \item  The modules $N$ and $M$ are $\delta$-\emph{interleaved} when there exist two morphisms  $u, v$ respectively from $M$ to $N^{\delta}$ and from $N$ to $M^{\delta}$ such that the right diagram in \Cref{fig:commute} commutes. 

    \item The \textit{interleaving distance} between $M$ and $N$ is defined as:
\[d_I(M,N) \coloneqq \inf \{ \delta \in \R^+ \tq M \text{ and } N \text{ are } \delta \text{-interleaved}\}. \]
\end{itemize}
\end{definition}

\begin{figure}[h!]
    \centering
\[\begin{tikzcd}
	{M_s} & {M_t} && {M_s} & {} & {M_{s+\delta}} && {M_{s+2\delta}} \\
	{N_s} & {N_t} && {N_s} & {} & {N_{s+\delta}} && {N_{s+2\delta}}
	\arrow["{j_s}", from=2-1, to=1-1]
	\arrow["{j_t}"', from=2-2, to=1-2]
	\arrow[from=1-1, to=1-2]
	\arrow[from=2-1, to=2-2]
	\arrow[from=1-4, to=1-6]
	\arrow[from=2-4, to=2-6]
	\arrow[from=2-6, to=2-8]
	\arrow[from=1-6, to=1-8]
	\arrow["{v_s}"{description, pos=0.25}, from=2-4, to=1-6]
	\arrow["{u_s}"{description, pos=0.25}, from=1-4, to=2-6]
	\arrow["{v_{s+\delta}}"{description, pos=0.25}, from=2-6, to=1-8]
	\arrow["{u_{s+\delta}}"{description, pos=0.25}, from=1-6, to=2-8]
\end{tikzcd}\]
    \caption{Commutative diagrams in the definition of morphisms (left) and $\delta$-interleaving (right).}
    \label{fig:commute}
\end{figure}

\begin{definition}[Bottleneck distance]
 A $\delta$-matching between two persistence diagrams $D, D'$ is a bijective map  $\gamma : C \to C'$, between subsets of $D, D'$ such that for any $c \in C$, $\norme{\gamma(c) - c}_{\infty} \leq \delta$, and such that for any $(a,b) \in (D \setminus C) \cup (D' \setminus C')$, $\module{a-b} \leq 2\delta$.

 The \emph{bottleneck distance} between two diagrams $D, D'$ is defined as:
\[ d_B(D, D') \coloneqq \inf \{ \delta \tq \text{There exists a } \delta \text{-matching between } D \text{ and } D' \}.\]

\end{definition}

A classical result in persistence theory is that the interleaving distance (of algebraic nature) between two persistence modules coincides with the bottleneck distance (of combinatorial nature) between their associated persistence diagrams.   

\begin{theorem}[Isometry theorem {\cite[Chapter 5]{StructureStability}}]
For any pair of pointwise finite-dimensional persistence modules $M,N$\! , we have:
\[ d_I(M,N) = d_B(\dgm(M), \dgm(N)).\] 
\end{theorem}

\section{Image Persistence}
\label{sec:image_persistence}

\begin{definition}[Image Persistence]
\label{def:image_pers}
Let $M, N$ be two persistence modules and ${f : M \to N}$ be a morphism between them. The \emph{image persistence module $\im f$ is the persistence module with vector space} ${(\im f)_t = f(M_t)}$ at value $t \in \R$ and whose connecting maps $(\im f)_a \to (\im f)_b $ are the restrictions of $N_a \to N_b$ to $(\im f)_a$ for every pair $a \leq b \in \R$. 
\end{definition}

We will only deal with image persistence modules arising in the following situation. Let $\phi : \R^d \to \R$ and $A \subset B \subset \R^d$. The homology with coefficients in $\mathbb{K}$ of the sublevels of $\phi$ restricted to $A$ or $B$, denoted by $A_a = A \cap \phi^{-1}(-\infty,a]$ and ${B_a = B \cap \phi^{-1}(-\infty, a]}$, together with the inclusion-induced maps $H_j(A_a) \to H_j(A_{a'})$ or  $H_j(B_a) \to H_j(B_{a'})$, yield two persistence modules in every degree $0 \leq j \leq d$. Furthermore, the inclusion $A_a \subset B_a$ yields for every degree $j$ a morphism of persistence modules $\iota^{\bullet}_j$:
\[ 
 \iota^{\bullet}_j : H_j(A_{\bullet}) \to H_j(B_{\bullet}). 
\]  
Assuming that each $\im \, \iota^{\bullet}_j$ is pointwise finite-dimensional, we write
$\dgm(\phi, A, B)$ for the graded persistence diagrams obtained by taking the image modules in every degree, and we call such diagrams \textit{image persistence diagrams}. To define the bottleneck distance between such diagrams, we only consider matchings that respect degrees.  The \textit{Euler characteristic} $\chi(\dgm(\phi, A, B)(r))$ is the alternating sum of the ranks of $\iota^{\bullet}_j$ at filtration value $r$. When $r$ is not a bound of an interval in $\dgm(\phi, A, B)$, this coincides with the alternating (in the degree) sum of the number of bars containing $r$, so that the notation $\chi(\dgm(\phi, A, B)(r))$ is justified almost everywhere.

Now we use results from Bauer \& Lesnick \cite{InducedMatchings} to show that a persistence module sandwiched between two persistence modules cannot be smaller in a certain sense than the image persistence module. We formalize that by saying that a persistence diagram $D'$ \textit{injects into} another persistence diagram $D$ when there is an injective map $\phi : D' \to D$ that respects degree and such that
$\phi((a',b')) = (a,b)$ with $a \leq a'$ and $b' \leq b$ for all $(a',b')$ in $D'$.

\begin{theorem}[Sandwiched persistent modules]
\label{thm:bars}
Let $A, B, Z$ be persistent modules such that $Z$ is pointwise finite-dimensional. Let furthermore be $\varphi, \psi$ morphisms of persistent modules and write $j = \psi \circ \varphi$:
\[
\begin{tikzcd}
A \ar[r, "\varphi"] \ar[rr, bend left, "j"] & Z \ar[r, "\psi"]& B
\end{tikzcd} \]

Then $\dgm(\mathrm{im} \, j)$ injects into $\dgm(Z)$. 
\end{theorem}

\begin{proof}
 The morphisms of persistence modules $im (\varphi) \to Z$, $im (\varphi) \to im (j)$ are respectively monomorphism and an epimorphism of persistence modules. By Lemma 4.2 in \cite{InducedMatchings}, we know that there exist injections of diagrams $\dgm(\im \, \varphi) \hookrightarrow \mathcal \dgm(Z)$ and $\dgm(\im \, j) \hookrightarrow \dgm( \im \, \varphi)$ respectively extending the intervals to the left and to the right.
\end{proof}

\begin{remark}[Number of bars]
\label{rm:number_of_bars}
For a persistence diagram $D$ and for two real numbers $a < b$, define $N_a^b(D)$ to be the total number of bars of $D$ intersecting with $[a,b]$. The theorem above implies that, with the same notations, $N_a^b(\dgm(\im j)) \leq N_a^b(\dgm(B))$. This is a generalization to persistence modules of the fact that the rank of a linear map cannot exceed the dimension of a vector space it factors through.
\end{remark}

We are now in position to prove a stability theorem for image persistence modules associated with sublevel set filtrations of locally Lipschitz functions.

\begin{theorem}[Image persistence stability theorem]
\label{thm:stability}
Let $h, \tilde{h} : \R^d \to \R$ be two real-valued function such that $\Vert h - \tilde{h} \Vert_{\infty} \leq \eps$ and $f: \R^d \to \R$ be a $\kappa$-Lipschitz function. Denote $\tilde{X}(u) = \tilde{h}^{-1}( - \infty, u]$ and $X(u) = h^{-1}(-\infty, u]$ for any real $u$. Moreover, let $\tilde{X}(u)_a = \tilde{X}(u) \cap f^{-1}(-\infty, a]$ and $X(u)_a = X(u) \cap f^{-1}(-\infty, a]$ for any reals $a,u$. Suppose that there exists $\mu > 0$ such that on $X(2\eps) \setminus X(-2\eps)$, $h$ is locally Lipschitz and $\Delta(\clarke h(x)) \geq \mu$  for every $x$ in $X(2\eps) \setminus X(-2\eps)$. Whenever $\mathrm{dgm}(f|_{X(0)})$ is defined,
\begin{equation*}
    d_B(\mathrm{dgm}(f, \tilde{X}(-\eps), \tilde{X}(\eps)), \mathrm{dgm}(f|_{X(0)})) \leq \frac{2 \kappa \eps}{\mu}.
\end{equation*}
\end{theorem}

\begin{proof}
This is an extension of the stability theorem for noisy domains from \cite{PersistentImage} to Lipschitz functions. We adapt this proof to our setting using approximate inverse flows of Lipschitz functions obtained in \Cref{thm:flow}.
For any $\sigma > 0$, take $C^{\sigma}_h(\cdot, \cdot)$ a continuous deformation retraction between $X(2\eps)$ and $X(0)$ given by \Cref{thm:flow}. Let $x \in X(2\eps)$. Each trajectory $C^{\sigma}_h(\cdot,x)$ is $\frac{2\eps}{\mu -\sigma}$-Lipschitz and needs at most time $1$ to send $x$ to $X(0)$. For every real $a$, $C^{\sigma}_h(1, \cdot) : X(2\eps)_a \to X(0)_{a + c}$ is a continuous map with $c \coloneqq \frac{2\kappa\eps}{\mu - \sigma}$. With the same reasoning we obtain a continuous map $X(0)_a \to X(-2\eps)_{a + c}$. The homotopies induced by the flows yield the following commutative diagram, where the vertical and horizontal maps are induced by inclusions:

\[
\begin{tikzcd}
	{} &
	{H_*(X(2\varepsilon)_a)} &
	{H_*(X(2\varepsilon)_{a+c})} &
	{H_*(X(2\varepsilon)_{a+2c})} & {} \\
	{} &
	{H_*(\tilde{X}(\varepsilon)_a)} &
	{H_*(\tilde{X}(\varepsilon)_{a+c})} &
	{H_*(\tilde{X}(\varepsilon)_{a+2c})} & {} \\
	{} &
	{H_*(X(0)_a)} &
	{H_*(X(0)_{a+c})} &
	{H_*(X(0)_{a+2c})} & {} \\
	{} &
	{H_*(\tilde{X}(-\varepsilon)_a)} &
	{H_*(\tilde{X}(-\varepsilon)_{a+c})} &
	{H_*(\tilde{X}(-\varepsilon)_{a+2c})} & {} \\
	{} &
	{H_*(X(-2\varepsilon)_a)} &
	{H_*(X(-2\varepsilon)_{a+c})} &
	{H_*(X(-2\varepsilon)_{a+2c})} & {}
	\arrow[dashed, from=1-1, to=1-2]
	\arrow[from=1-2, to=1-3]
	\arrow[draw={rgb,255:red,214;green,92;blue,92}, from=1-2, to=3-3, very thick, dashed]
	\arrow[from=1-3, to=1-4]
	\arrow[from=1-3, to=3-4]
	\arrow[dashed, from=1-4, to=1-5]
	\arrow[dashed, from=2-1, to=2-2]
	\arrow[draw={rgb,255:red,214;green,92;blue,92}, from=2-2, to=1-2, very thick, dashed]
	\arrow[from=2-2, to=2-3]
	\arrow[from=2-3, to=1-3]
	\arrow[from=2-3, to=2-4]
	\arrow[from=2-4, to=1-4]
	\arrow[dashed, from=2-4, to=2-5]
	\arrow[dashed, from=3-1, to=3-2]
	\arrow[from=3-2, to=2-2]
	\arrow[from=3-2, to=3-3]
	\arrow[from=3-2, to=5-3]
	\arrow[from=3-3, to=2-3]
	\arrow[from=3-3, to=3-4]
	\arrow[draw={rgb,255:red,92;green,92;blue,214}, from=3-3, to=5-4, very thick, dotted]
	\arrow[draw={rgb,255:red,92;green,92;blue,214}, from=3-4, to=2-4, very thick, dotted]
	\arrow[dashed, from=3-4, to=3-5]
	\arrow[dashed, from=4-1, to=4-2]
	\arrow[from=4-2, to=3-2]
	\arrow[from=4-2, to=4-3]
	\arrow[from=4-3, to=3-3]
	\arrow[from=4-3, to=4-4]
	\arrow[draw={rgb,255:red,92;green,92;blue,214}, from=4-4, to=3-4, very thick, dotted]
	\arrow[dashed, from=4-4, to=4-5]
	\arrow[dashed, from=5-1, to=5-2]
	\arrow[from=5-2, to=4-2]
	\arrow[from=5-2, to=5-3]
	\arrow[from=5-3, to=4-3]
	\arrow[from=5-3, to=5-4]
	\arrow[draw={rgb,255:red,92;green,92;blue,214}, from=5-4, to=4-4, very thick, dotted]
	\arrow[dashed, from=5-4, to=5-5]
\end{tikzcd}\]
As in \cite{PersistentImage}, the two colored, respectively dashed and dotted arrow paths provide interleavings showing that, for every $\sigma \in (0, \mu)$:
\[
d_B(\dgm(f, \tilde{X}(-\eps), \tilde{X}(\eps)), \dgm(f|_{X(0)}) ) \leq \frac{2 \kappa \eps }{\mu - \sigma}.
\]
\end{proof}

Now let $X, Y$ be two compact sets of $\R^d$ and $f = d_x : z \mapsto \norme{z -x}$ be the distance function to any point $x$. Applying the previous theorem with $h = d_X + 2\eps$, $\tilde{h} = d_Y + 2\eps$ yields the following corollary:
\begin{corollary}[Image stability theorem for compact sets]
\label{cor:stability_compact}
Let $\mu \in (0,1]$, $\eps > 0$ and $X, Y$ be two compact subsets of $\R^d$ such that $d_H(X,Y) \leq \eps \leq \frac{1}{4} \reach_{\mu}(X)$. Then for any $x \in \R^d$, 
\begin{equation*}
    d_B(\dgm(d_x, Y^{\eps}, Y^{3\eps}), \dgm((d_x)|_{X^{2\eps}})) \leq \frac{2\eps}{\mu}.
\end{equation*}
\end{corollary}

\section{Persistent Intrinsic Volumes}
\label{sec:persistent_volumes}
In this section, we define the persistent intrinsic volumes $V^{\eps,R}_i(Y)$, where $\eps, R$ are positive numbers, and state our main results.

\begin{definition}[Persistent Intrinsic Volumes]
\label{def:persistent_volumes}
Let $Y$ be a compact subset of $\R^d$, $x \in \R^d$ and $\eps \geq 0$. 
We let $D^{\eps,x}_Y \coloneqq \dgm(d_x, Y^{\eps}, Y^{3\eps})$.
When $\eps = 0$, we write $D^{x}_Y \coloneqq \dgm((d_x)|_{Y})$.
\begin{itemize}
    \item The \emph{persistent Steiner function} $Q^{\eps}_Y$ is the function $[0, +\infty) \to \R$ defined by:
    \begin{equation*}
    Q^{\eps}_Y(r) \coloneqq \int_{x \in \R^d} \chi(D^{\eps,x}_Y(r)) \mathrm{d} x.
    \end{equation*}
    (see \Cref{def:image_pers} for the definition of the Euler characteristic of an image persistence diagram).
    \item Given $R >0$, the persistent Steiner polynomial is the orthogonal projection of $Q^{\eps}_Y$ restricted to $[0,R]$ on the space of polynomials of degree at most $d$ for the inner product of $L^2\left ([0,R] \right)$.
    \item Writing the persistent Steiner polynomial as $\sum_{i=0}^d \omega_{i}V^{\eps, R}_{d-i}(Y) r^i$, the quantities $(V^{\eps, R}_i(Y))_{0 \leq i \leq d}$ are the \emph{persistent intrinsic volumes} of $Y$.
\end{itemize}
\end{definition}

The following lemma shows that the persistent Steiner function as well
as the persistent intrinsic volumes are always well-defined when working with Čech homology. Čech homology may differ from singular homology, but both agree on sufficiently nice spaces such as paracompact, locally contractible Hausdorff spaces. It has the convenient property of being continuous with respect to inverse limits.

\begin{lemma}[Persistent intrinsic volumes are well-defined]
\label{lem:well-defined}
Let $\eps, r > 0$ and denote by $\phi$ the map $\R^d \to \mathbb{Z}, x \mapsto \chi(D^{\eps, x}_Y(r))$ (that is, the integrand in the first point of \Cref{def:persistent_volumes}). Then the following holds:
\begin{enumerate}
    \item The map $\phi$ is bounded;
    \item The map $\phi$ in supported in $Y^{r+\eps}$;
    \item If we use Čech homology to define persistence modules, $\phi$ is measurable.
\end{enumerate}
As a consequence, the persistent Steiner function $Q^{\eps}_Y$ is defined for any compact $Y \subset \R^d, \eps > 0$, measurable and bounded on every compact interval.
\end{lemma}
\begin{proof}
Since $Y$ is compact and $\eps$ is positive, one can always find a finite simplicial complex $C$ such that $Y^{\eps} \subset C \subset Y^{3\eps}$. By \Cref{thm:bars} the diagrams $D^{\eps,x}_Y$ inject into $D^x_C$ for every $x$, implying that their size is bounded by the number of simplices of $C$, which proves assertion (1). Remark that $d_Y(x) \geq \eps + r$ implies $B(x,r) \cap Y^{r} = \emptyset$ which proves assertion (2). 

Now there remains to prove assertion (3).
Consider a family of finite sets $(P_{\delta})_{\delta > 0}$ such that the Hausdorff distance between $Y$ and $P_{\delta}$ is less than $\delta$. Put $P^{t}_{\delta,r} = (P_{\delta})^{t} \cap B(x,r)$ and $Y^{\eps}_r = Y^{\eps} \cap B(x,r)$ when $x$ is clear from context. For any $\delta$ in $[0, \eps]$ and for any $x \in \R^d$, inclusions given by the Hausdorff distance bound induce maps $\iota_{\delta}, \iota_{\delta}^+, \iota_{\delta}^-, \iota_{\delta}^*$ whose domain and codomain can be read in the following commutative diagram:


\begin{adjustwidth}{-50pt}{-50pt}
\[\begin{tikzcd}[sep= 1.6 em]
	{H_i(Y^{\varepsilon}_r)} & {H_i(P^{\varepsilon +\delta}_{\delta,r})} & {H_i(Y^{\varepsilon + 2\delta}_r)} & {H_i(Y^{3\varepsilon}_r)} & {H_i(P^{3\varepsilon +\delta}_{\delta,r})} & {H_i(Y^{3\varepsilon+2\delta}_r)}
	\arrow[from=1-1, to=1-2]
	\arrow["{\iota_*}"', bend right = 13, from=1-1, to=1-4]
	\arrow["{\iota^{-}_{\delta}}", bend left = 15, from=1-1, to=1-6]
	\arrow[from=1-2, to=1-3]
	\arrow["{\iota_{\delta}}", bend left = 13, from=1-2, to=1-5]
	\arrow["{\iota^+_{\delta}}", from=1-3, to=1-4]
	\arrow[from=1-4, to=1-5]
	\arrow[from=1-5, to=1-6]
\end{tikzcd}\]
\end{adjustwidth}

 Factorizations of linear maps imply that $\rank(\iota^-_{\delta}) \leq \rank(\iota_{\delta}) \leq \rank(\iota^+_{\delta})$. 
The Čech homology functor is continuous with respect to inverse limits. In particular, since $Y^{\eps}_r$ (resp. $Y^{3\eps}_r$) is the limit of $\cap_{\delta > 0} Y^{\eps + \delta}_r$ (resp. $\cap_{\delta > 0} Y^{3\eps + \delta}_r$) as a nested system \cite[p. 261]{EilenbergSteenrod}, for any fixed $x$ the quantity  $\rank(\iota^+_{\delta})$ (resp. $\rank(\iota^-_{\delta})$) tends to $\rank(\iota_*)$ as $\delta$ tends to 0. We thus end up with 
\[ \lim_{\delta \to 0} \rank(\iota_{\delta}) = \rank(\iota_*). \]
Now, thanks to the semi-algebraic nature of both balls and finite unions of balls, the maps $\rank(\iota_{\delta})$ are definable as functions of $x$ \cite[p. 147]{TameTopology}, hence measurable.
As the pointwise limit of the measurable functions $\rank(\iota_{\delta})$, $\rank(\iota_*)$ is measurable as well.
Reasoning as in the proof of $(1)$, the map $\rank(\iota_*)$ is bounded by the number of simplices of $C$. The map $\phi$ is the alternating sum of $\rank(\iota_{*})$, hence the claim follows.

\end{proof}

\begin{remark}
Since intersections of $Y^{\eps}$ with balls of radius smaller than $\reach(Y^{\eps})$ are either contractible or empty \emph{\cite[4.15]{FedererCurvature}}, the persistent Steiner function $Q^{\eps}_Y$ coincides with the Steiner polynomial $Q_{Y^{\eps}}$ on $[0, \reach(Y) - \eps)$.
\end{remark}

The following is an immediate consequence of \Cref{cor:stability_compact} and \Cref{thm:bars}.

\begin{proposition}[Diagram approximation]
\label{prop:approx_diag}
Let $X,Y$ be compact subsets of $\R^d$ and $\eps, \mu > 0$ be such that $d_H(X,Y) \leq \eps \leq \frac{1}{4} \reach_{\mu}(X)$. For all $x \in \R^d$, we have:
\begin{itemize}
    \item $d_B(D^{\eps,x}_{Y}, D^{x}_{X^{2\eps}}) \leq \frac{2 \eps}{\mu}$;
    \item $D^{\eps, x}_Y$ injects into $D^{x}_{X^{2\eps}}$.
\end{itemize}
\end{proposition}

The next lemma bounds the average difference of the Euler characteristics of close persistent diagrams.

\begin{lemma}
[$\chi$-averaging lemma]
\label{lem:averaging}
Let $D, D'$ be two persistence diagrams with $d_B(D, D') \leq \eps$ and let $a < b \in \R$. Recall that $N_a^b(D)$, $N_a^b(D')$ denote the number of intervals in diagrams $D$, $D'$ intersecting interval $[a,b]$. Then we have: 
\begin{equation*}
 \int_a^b \module{\chi(D(t)) - \chi(D'(t))} \mathrm{d}t \leq 2 \eps (N_a^b(D) + N_a^b(D')).
\end{equation*}
 In particular, if $D$ injects into $D'$
 we have:
\begin{equation*}
 \int_{a}^b  \module{\chi(D(t)) - \chi(D'(t))} \mathrm{d}t \leq 4 \eps N_a^b(D').
\end{equation*}
\end{lemma}

\begin{proof}
 The first inequality is a slight extension of an argument obtained in \cite{CurveInequalities}.
Let $\mathcal{I}_i$ (resp. $\mathcal{I}'_i$) be the set of intervals of $D$ (resp. $D'$) in degree $i$, and recall that $\mathcal{I}_i$, $\mathcal{I}'_i$ are at bottleneck distance at most $\eps$.
We have:
\[ \chi(D(t)) = \sum_{i=0}^{d} (-1)^{i} \sum_{I_i \in \mathcal{I}_{i}} \one_{I_i}(t). \]
and similarly for $\chi(D'(t))$. 
Let $\gamma$ be an $\eps$-matching from $D$ to $D'$.  Define $C_i$ and $C'_i$ to be the respective largest subsets of $\mathcal{I}_i$ and $\mathcal{I}'_i$ matched bijectively by $\gamma$. We have:
\begin{align*}
    \int_a^b \module{\chi(D(t)) - \chi(D'(t))} \diff t \leq  \sum_{i=0}^d & \Bigl ( 
    \sum_{I_i \in C_i}   \int_{a}^b  \module{\one_{I_i} - \one_{\gamma(I_i)}} 
     \diff t  \\ & + \hspace{-7 pt} \sum_{I_i \in \mathcal{I}_i \setminus  C_i} \int_a^b\one_{I_i}(t) \diff t + \sum_{I'_i \in \mathcal{I}'_i \setminus  C'_i} \int_{a}^b\one_{I'_i}(t) \diff t  \Bigr ).
\end{align*}
Since $\gamma$ is an $\eps$-matching, each of these integrals is bounded by $2\eps$,
either because there are two parts of length at most $\eps$ where matched intervals do not overlap or because the length of non-matched intervals is at most $2\eps$. 
Now for the sum over $C_i$,  we have $\mathrm{supp} \module{\one_{I_i} - \one_{\gamma(I_i)}} \subset I_i \cup \gamma(I_i)$, meaning that the integral over $[a,b]$ vanishes if none of these intervals intersect with $(a,b)$. 
Similarly, for the second and third sums, only intervals intersecting with $(a,b)$ contribute to the sum.
Overall, there is at most one non-vanishing contribution per interval in $D \cup D'$ whose intersection with $(a,b)$ is non-empty, and no contribution otherwise. The total number of non-vanishing integrals is bounded by $N_a^b(D) + N_a^b(D')$.
\end{proof}

We use the previous lemmas to establish a linear rate of convergence  of
the persistent Steiner function $Q^{\eps}_Y$ to the Steiner polynomial $Q_{X^{2\eps}}$ over $[0,R]$ for any positive real $R$.

\begin{theorem}[Approximating the Steiner polynomial in the $L^1$ norm]
\label{thm:steiner_rate}
Let $X,Y \subset \R^d$ be compact sets and $\eps, \mu > 0$ be such that $d_H(X,Y) \leq \eps \leq \frac{1}{4}\reach_{\mu}(X)$. Then we have:
\begin{equation*}
\norme{Q^{\eps}_Y - Q_{X^{2\eps}}}_{L^1([0,R])} \leq \frac{8\eps}{\mu}\int_{\R^d} N_0^R(D^{x}_{X^{2\eps}}) \mathrm{d} x.
\end{equation*}
\end{theorem}
\begin{proof}

Indeed, we have the following:
\begin{align*}
\left\|Q_Y^\varepsilon - Q_{X^{2\eps}}\right\|_{L^1([0,R])}
&= \int_0^R \left|Q_Y^\varepsilon(r) - Q_{X^{2\eps}}(r) \right|\,dr \\
&= \int_0^R \left|
\int_{\mathbb{R}^d} \chi\!\left(D_Y^{\varepsilon,x}(r)\right)\,dx
-
\int_{\mathbb{R}^d} \chi\!\left(D_{X^{2\eps}}^{x}(r)\right)\,dx
\right|\,dr \\
&\leq \int_0^R \int_{\mathbb{R}^d}
\left|
\chi\!\left(D_Y^{\varepsilon,x}(r)\right)
-
\chi\!\left(D_{X^{2\eps}}^{x}(r)\right)
\right|
\,dx\,dr \\
&\leq \int_{\mathbb{R}^d}
\frac{8\varepsilon}{\mu}\,
N_0^R\!\left(D_{X^{2\eps}}^{x}\right)\,dx,
\end{align*}
where the last line follows from \Cref{lem:averaging} and the bottleneck bound of \Cref{prop:approx_diag}.
\end{proof}

The next lemma relates the bound of previous \Cref{thm:steiner_rate} to the $R$-curvature mass of $X^{2\eps}$.

\begin{lemma}[Number of critical points of $d_x$]
\label{lemma:number_crit}
Let $X \subset \R^d$ be a compact set and $\eps, \mu >0$ be such that $0 < 2\eps < \reach_{\mu}(X)$ (which implies that $X^{2\eps}$ admits a normal bundle). For any $R > 0$, define $\pi_{X^{2\eps}, R}: \Nor(X^{2\eps}) \times [0,R] \to \R^d, (y,n,t) \mapsto y + tn$.\\
Then  $\mathcal{H}^{d}$-almost everywhere in $x$,
 \begin{itemize}
     \item $\mathrm{card} \; \pi_{X^{2\eps}, R}^{-1}(x)$ is finite;
     \item  $N_0^R(D^x_{X^{2\eps}}) \leq \one_{X^{2\eps}}(x) + \mathrm{card} \; \pi_{X^{2\eps},R}^{-1}(x)$.
 \end{itemize}
 Moreover, we have:
\[ \int_{\R^d} N_0^R(D^x_{X^{2\eps}}) \diff x \leq \Vol(X^{2\eps}) + M_R(X^{2\eps}), \]
where $M_R(X^{2\eps})$ is the $R$-curvature mass defined in \Cref{def:Rmass}.
\end{lemma}

\begin{proof}
The map $\pi_{X^{2\eps},R}$ is a Lipschitz map between two rectifiable sets of dimension $d$. By the coarea formula of Federer \cite[Theorem 3.2.22]{GMT}, denoting by $J_d \pi_{X^{2\eps}, R}$ the $d$-dimensional Jacobian of $\pi_{X^{2\eps}, R}$, we have: 
\begin{equation}
\label{eq:nombre_barres}
\int_{\R^d} \mathrm{card} \; \pi_{X^{2\eps}, R}^{-1}(x) \diff x = \int_{\Nor(X^{2\eps}) \times [0,R]} J_d \pi_{X^{2\eps}, R} (y,n,t) \diff \mathcal{H}^{d}(y,n,t).
\end{equation}
Since $\pi_{X^{2\eps}, R}$ is Lipschitz on the compact set $\Nor(X^{2\eps}) \times [0,R]$, this quantity has to be finite, yielding the finiteness of $\mathrm{card} \; \pi_{X^{2\eps}, R}^{-1} (x)$ almost everywhere. 

 Let $x \in \R^d$ and let $X^{2\eps}_t \coloneqq X^{2\eps} \cap d_x^{-1}(-\infty, t]$ be the $t^2$-sublevel sets of ${f_x \coloneqq (d_x)^2|_{X^{2\eps}}}$. When $t < 0$, the set $X^{2\eps}_t$ is empty. When $t = 0$, this set is either a point when $x \in X^{2\eps}$, or empty.
  The function $f_x : X^{2\eps} \to \R$ cannot be Morse in the classical sense of Milnor \cite{milnor_morse} as $X^{2\eps}$ is not a submanifold of $\R^d$. Nevertheless by the $\mu$-reach assumption, $X^{2\eps}$ is an offset of $X$ at a regular value of its distance function, a setting to which the generalized Morse theory of \cite{MorseTubular} applies.
 Specifically, if $f_x$ 
 is a Morse function in the sense of \cite{MorseTubular}, the changes in homology of the filtration $(X^{2\eps}_t)_{t > 0}$ happen at the square roots of the critical values of $f_x$ and there is exactly one homology change per critical point of this map. The number of bars of a persistence diagram of a Morse function is bounded by its number of critical points and we have:
 \begin{equation*}
    \label{inequality:nombre_crit_1}
     N_0^R(D^x_{X^{2\eps}}) \leq \one_{X^{2\eps}}(x) + \mathrm{card} \Bigl \{ \text{Critical points of } f_x \text { with distance to } x \text{ in } (0,R] \Bigr \}.
 \end{equation*}

We prove that the set of points $x$ in $\R^d$ such that $f_x$ is a Morse function in the sense of this generalized Morse theory has full Lebesgue measure, and we relate the number of critical points of $f_x$ to the map $\pi_{X^{2\eps}, R}$. Following \cite[Definition 2.5]{MorseTubular}, a point $z \neq x$ is a critical point for $f_x$ with value in $(0,R^2]$ if and only if the pair $\left (z, \frac{z - x}{\norme{z-x}} \right) \eqqcolon (z,n)$ belongs to the normal bundle $\Nor(X^{2\eps})$ and $\norme{x-z} \leq R$, that is, when $(z,n, \norme{x-z}) \in \pi_{X^{2\eps}, R}^{-1}(x)$. In case $f_x$ is Morse, the previous inequality becomes:
\begin{equation*}
N_0^R(D^x_{X^{2\eps}}) \leq \one_{X^{2\eps}}(x) + \mathrm{card} \;\pi_{X^{2\eps}, R}^{-1}(x).
\end{equation*}
We shall prove that $\mathcal{H}^d$-almost everywhere in $x$, the critical points of $f_x$ are non-degenerate.
Let $M$ be the set of points $x$ such that $\pi_{X^{2\eps}, R}^{-1}(x)$ is infinite and let $\mathrm{Bad}(X^{2\eps})$ be the set of non-regular pairs of $\Nor(X^{2\eps})$. The set ${\mathrm{reg}(X^{2\eps}) \coloneqq \R^d \setminus (M \cup \pi_{X^{2\eps}, R}(\mathrm{Bad}_{X^{2\eps}}\times [0,R]) \cup \partial X^{2\eps})}$ has full Lebesgue measure.

Take $x \in \mathrm{reg}(X^{2\eps})$ and let $(b_i)_{1 \leq i \leq m}$ be an orthonormal basis of principal directions of $\Tan \left (\Nor(X^{2\eps}), \left (z,\frac{z - x}{\norme{z-x}} \right )\right )$ associated with the principal curvatures $(\kappa_i)_{1 \leq i \leq m}$.
Let $u = \sum_{i =1}^m \lambda_i b_i$, $v = \sum_{i=1}^m \nu_i b_i$ be any pair in $\Tan \left (\Nor(X^{2\eps}), \left (z,\frac{z - x}{\norme{z-x}} \right) \right )$ and $r = \norme{z - x}$. Following notations from \cite[Definition 2.5]{MorseTubular}, the Hessian of $f_x$ at $z$ is: 
\begin{equation*}
(H_{|X^{2\eps}}f_x)(z): (u,v) \mapsto \sum_{i=1}^m (2 + 2r\kappa_i)\lambda_i \nu_i.
\end{equation*} 
This bilinear form is degenerate if and only if $\displaystyle r = -\frac{1}{\kappa}$ where $\kappa$ is a principal curvature of a point in $\pi_{X^{2\eps}, R}^{-1}(x)$. There are at most $d$ distinct principal curvatures, from which we know that the set 
\begin{align*}
\mathrm{Degen}(X^{2\eps}) \coloneqq \left \{ \Bigl (z, n, -\frac{1}{\kappa} \right )&  \bigg | \; (z,n) \in \Nor(X^{2\eps}) \text{ is regular,}\\
& \kappa \text{ is a principal curvature at } (z,n) \Bigr \}. 
\end{align*}
has Hausdorff dimension at most $(d-1)$. Thus $\mathrm{reg}(X^{2\eps}) \setminus \pi_{X^{2\eps}, R}(\mathrm{Degen}(X^{2\eps}))$ is a set of full Lebesgue measure whose elements $x$ satisfy $N_0^R(D^{x}_{X^{2\eps}}) \leq \one_{X^{2\eps}}(x) + \mathrm{card} \; \pi_{X^{2\eps}, R}^{-1}(x)$.

Now by the structure of tangent spaces (\Cref{prop:structure}) we have:  
\[ J_d \pi_{X^{2\eps}, R} (x,n,t)  = \prod_{i=1}^{d-1} 
\frac{\module{1 + t \kappa_i(x,n)}}{\sqrt{1 + \kappa_i(x,n)^2}}.\]
and thus by the co-area formula \Cref{eq:nombre_barres},
\begin{equation*}
    \int_{\R^d} N_0^R(D^x_{X^{2\eps}}) \leq \int_{\R^d} \left ( \one_{X^{2\eps}}(x) +  \mathrm{card} \; \pi_{X^{2\eps}, R}^{-1}(x) \right )\diff x = \Vol(X^{2\eps}) + M_R(X^{2\eps}).
\end{equation*}

\end{proof}

We are now in position to prove our main theorems.

\begin{theorem}[Linear convergence of the persistent intrinsic volumes]
\label{thm:persistent_rate}
Let $X,Y$ be compact subsets of $\R^d$ and let $\eps, \mu > 0$ be such that $d_H(X,Y) \leq \eps \leq \frac{1}{4} \reach_{\mu}(X)$. There exist constants $P(i,d)$ such that for any $R > 0$, we have:
\begin{equation*}
\label{eq:borne_volumes}
    \module{V^{\eps,R}_i(Y) - V_i(X^{2\eps})} \leq \frac{\eps P(i,d)}{\mu R^{d-i+1}} \left (\Vol(X^{2\eps}) + M_R(X^{2\eps}) \right ).
\end{equation*}
\end{theorem}

\begin{proof}
 Applying \Cref{thm:steiner_rate} and \Cref{lemma:number_crit} we obtain for any positive $R$:
 \[ \norme{Q^{\eps}_Y - Q_{X^{2\eps}}}_{L^1([0,R])} \leq \frac{8\eps}{\mu} ( \Vol(X^{2\eps}) + M_R(X^{2\eps})).\]
 Let $i \leq j$ and let $(L_n)_{n \in \N}$ be the Legendre polynomials on $[0,1]$ - recall that they form an orthonormal basis of the space of polynomials equipped with the inner product of $L^2([0,1])$. We obtain after renormalization and reparametrization an orthonormal basis of the polynomials of degree at most $d$ on $[0,R]$ formed by the following polynomials: \[ P^R_j(X) \coloneqq  \sqrt{\frac{2j+1}{R}} L_j\left (\frac{X}{R}\right ). \] The fact that $\norme{L_i}_{\infty, [0,1]} \leq 1$ yields $\norme{P^R_j}_{\infty, [0,R]} \leq \sqrt{\frac{2j +1}{R}}$ for any positive $R$.

Denoting by $c_i(P^R_j)$ the $i$-th coefficient of $P^R_j$, we have for any $i \leq j$:
\[ \module{c_i(P^R_j)} = \frac{1}{R^{i+\frac{1}{2}}} \sqrt{2j +1} \binom{j}{i} \binom{i+j}{i} \]
and $c_i(P^R_j) = 0$ if $i > j$. Now the persistent Steiner polynomial is the
orthogonal projection of $Q^{\eps}_Y$ on the space of polynomials of degree at most $d$. Decomposing it in the basis $P^R_j$ and taking the coefficient of $X^{d-i}$ yields
\[      V^{\eps, R}_i(Y) =  \frac{1}{\omega_{d-i}} \int_0^R \int_{\R^d} \chi(D^{\eps,x}_Y(r)) \sum_{j = d-i}^d P^R_j(r) c_{d-i}(P^R_j)  \diff x \diff r.\]
Since we have
\[
     \norme{P^R_j c_{d-i}(P^R_j)}_{\infty, [0,R]} \leq  \frac{(2j+1)}{R^{d-i+1}}  \binom{j}{d-i} \binom{d-i+j}{j},
     \]
we see that:
\begin{equation*}
\module{V^{\eps,R}_i(Y) - V_i(X^{2\eps})} \leq \frac{1}{R^{d-i+1}}  \norme{Q^{\eps}_Y - Q_{X^{2\eps}}}_{L^1([0,R])} \frac{1}{\omega_{d-i}}\sum_{j = d-i}^d (2j+1) \binom{j}{d-i} \binom{d-i+j}{j}.    
\end{equation*}
We now obtain the desired inequality by applying \Cref{lemma:number_crit} and \Cref{thm:steiner_rate} and putting \[ P(i,d) \coloneqq \frac{8}{\omega_{d-i}}\sum_{j = d-i}^d (2j+1) \binom{j}{d-i} \binom{d-i+j}{j}. \]

\end{proof}

\begin{remark}
Our approach consists in bounding the $L^1$ norm of $Q^{\eps}_Y - Q_{X^{2\eps}}$ (\Cref{thm:steiner_rate} and \Cref{lemma:number_crit}):
\[ \norme{Q^{\eps}_Y - Q_{X^{2\eps}}}_{L^1([0,R])} \leq \frac{8\eps}{\mu} \left ( \Vol(X^{2\eps}) + M_R(X^{2\eps})\right ), \] and retrieving quantities that are close to the rescaled coefficients $(V_i(X^{2\eps}))_{0 \leq i \leq d}$ of the polynomial $Q_{X^{2\eps}}$ from the persistent Steiner function $Q^{\eps}_Y$, which is in $L^{2}([0,R])$ but is not a priori a polynomial.
We chose in \Cref{def:persistent_volumes} the coefficients $V^{\eps,R}_i(Y)$ of its orthogonal projection on the space $\R_d[X]$ of polynomials of degree at most $d$. From this definition and from the explicit formulas for Legendre polynomials (see the proof of \Cref{thm:persistent_rate}) we obtain: 
\begin{equation}
    \label{eq:equation_6.8}
    \module{V^{\eps, R}_i(Y) - V_i(X^{2\eps})} 
    \leq \frac{ P(i,d)}{8 R^{d-i+1}} \norme{Q^{\eps}_Y - Q_{X^{2\eps}}}_{L^1([0,R])}.
\end{equation}

Taking the $i$-th coefficients of orthogonal projections is one way to build linear forms $\phi^R_{i,d} : L^{2}([0,R]) \to \R$ such that, restricted to the space $\R_d[X]$ of polynomials with degree at most $d$, $\phi^R_{i,d}$ is the map $\sum_{0 \leq j \leq  d} a_j X^j \mapsto a_i$. In fine we defined $V_i^{\eps,R}(Y)$ as $\phi^R_{d-i,d}(Q^{\eps}_Y)/\omega_{d-i}$. By \Cref{eq:equation_6.8} the linear forms $\phi^R_{i,d}$ are $\omega_{d-i}P(i,d)/8R^{i+1}$-Lipschitz for the $L^1$ norm over $[0,R]$. The bound we infer on $V^{\eps,R}_i(Y) - V_i(X^{2\eps})$ from $\norme{Q^{\eps}_Y - Q_{X^{2\eps}}}_{L^1([0,R])}$ comes from a bound on the Lipschitz constant of $\phi^{R}_{d-i,d}$ obtained using Legendre polynomials. Yet, by the Hahn-Banach extension theorem, the best Lipschitz constant possible for a linear form $\phi^R_{i,d}$ whose restriction to $\R_d[X]$ is the $i$-th coefficient map is exactly the Lipschitz constant of $(\phi^R_{i,d})|_{ \R_d[X]}$  which we denote by $l^R_{i,d}$. When $i = d \geq 0$ by classical work in optimization (e.g., \cite[4.9, p. 117]{borne_lip}), we have:
\[ l^{R}_{d,d} = \frac{4^{d}}{R^{d+1}} \leq \frac{(2d+1)}{R^{d+1}} \binom{2d}{d}, \]
where the right-hand side is the bound on the Lipschitz constant we explicitly obtained by the Legendre polynomials method. This shows that there exist alternate ways of defining the persistent intrinsic volumes from the persistent Steiner function that lead to a strictly better bound than the one in \Cref{thm:persistent_rate}.
\end{remark}

\begin{theorem}[Convergence of the intrinsic volumes of an offset]
\label{thm:ConvIntrOffset}
Let $X$ be a compact subset of $\R^d$ and $\mu, \eps > 0$ be such that $\eps < \frac{1}{2}\reach_{\mu}(X)$.
If $\reach(X) > 0$, we have:
\begin{equation*}
    \module{V_i(X) - V_i(X^{2\eps})} \leq \frac{\eps P(i,d)}{2\mu R^{d-i+1}} \left ( \Vol(X) + \Vol(X^{2\eps}) + M_R(X) + M_R(X^{2\eps}) \right ).
\end{equation*}

 When $X$ is subanalytic, this inequality holds
 when we replace $M_R(X)$ by $\displaystyle \liminf_{r \to 0} M_R(X^r)$.
\end{theorem}

\begin{proof}
 Let $x \in \R^d$ and $\delta > 0$. For any sufficiently small $\sigma > 0$, set $c = \frac{2\eps - \delta}{\mu - \sigma}$. By the $\mu$-reach hypothesis, there exists a continuous flow between $X^{2\eps}$ and $X^{\delta}$ which is $c$-Lipschitz in the time parameter thanks to \Cref{thm:flow}. This yields the following commutative diagram:
\[\begin{tikzcd}
	{} & {H_*(X^{2\varepsilon}_a)} & {H_*(X^{2\varepsilon}_{a + c})} & {H_*(X^{2\varepsilon}_{a + 2c})} & {} \\
	{} & {H_*(X^{\delta}_a)} & {H_*(X^{\delta}_{a + c})} & {H_*(X^{\delta}_{a + c})} & {}
	\arrow[from=2-2, to=1-2]
	\arrow[from=1-2, to=2-3]
	\arrow[from=2-2, to=2-3]
	\arrow[from=1-2, to=1-3]
	\arrow[from=2-3, to=1-3]
	\arrow[from=1-3, to=1-4]
	\arrow[from=2-3, to=2-4]
	\arrow[from=1-3, to=2-4]
	\arrow[dashed, from=1-1, to=1-2]
	\arrow[dashed, from=2-1, to=2-2]
	\arrow[dashed, from=1-4, to=1-5]
	\arrow[dashed, from=2-4, to=2-5]
	\arrow[from=2-4, to=1-4]
\end{tikzcd}\]
This gives a $\frac{2\eps - \delta}{\mu - \sigma}$ interleaving between the two persistence modules, which implies
$d_B(D^x_{X^{2\eps}}, D^x_{X^{\delta}}) \leq \frac{2\eps}{\mu}$ by letting $\sigma$ go to zero.
Using the same reasoning as in the proof of \Cref{thm:persistent_rate}, except that we have to use the first inequality of \Cref{lem:averaging} since a priori there is no injection between the two diagrams $D^x_X$ and $D^x_{X^{2\eps}}$, we get for any positive $R$:
\begin{equation*}
\label{eq:inegalite_minioffset}
    \module{V_i(X^{2\eps}) - V_i(X^{\delta})} \leq \frac{\eps P(i,d)}{2\mu R^{d-i+1}}  \int_{\R^d} (N_0^R(D^x_{X^{2\eps}}) + N_0^R(D^x_{X^{\delta}})) \diff x.
\end{equation*}
 By \Cref{lemma:number_crit}:
\begin{equation}
\label{eq:inegalite_MR}
    \module{V_i(X^{2\eps}) - V_i(X^{\delta})} \leq \frac{\eps P(i,d)}{2 \mu R^{d-i+1}}  \left ( \Vol(X^{2\eps}) + M_R(X^{2\eps}) + \Vol(X^{\delta}) + M_R(X^{\delta })\right). 
\end{equation}
We already know that $\Vol(X^{\delta})$ converges to $\Vol(X)$. 
Let us prove the first statement and assume that $\reach(X) > 0$. By the tube formula, the intrinsic volumes of $X^{\delta}$ are the (scaled) coefficients of the Steiner polynomial of $X$ translated by $\delta$, and thus $V_i(X^{\delta})$ converges  to $V_i(X)$ when $\delta$ goes to zero.
We are left to prove that $\lim_{\delta \to 0} M_R(X^{\delta}) = M_R(X)$ to obtain the desired inequality. If $0 \leq \delta < \reach(X)$, writing $h_{\delta} : (x,n) \mapsto (x + \delta n, n)$, we have:
\[ \Nor(X^{\delta}) = \bigcup_{(x,n) \in \Nor(X)} \{ x +\delta n \} \times \{ n \} = h_{\delta}(\Nor(X)).\]

Denoting by $\kappa_i(z,n)$ the principal curvatures of $X$ at a regular pair $(z,n)$ of $X$ and $\kappa_{\delta,i}(z',n')$ the principal curvatures of $X^{\delta}$ at a regular pair $(z',n')$ of $\Nor(X^{\delta})$, we have: \[\kappa_{\delta,i}(z + \delta n, n) = f_{\delta}(\kappa_i(z,n)), \] where $f_{\delta}(s) = \frac{s}{1 +\delta s}$. 
Since $\kappa_i(x,n) \geq - \frac{1}{\reach(X)}$ the quantities $f_{\delta}(\kappa_i(x,n))$ are well-defined for any $\delta < \reach(X)$. The change of variable formula yields for any positive $t$:
\begin{align*}
M_R(A^{\delta}) = & \int_0^R
    \int_{\Nor(A^{\delta})} \prod_{i=1}^{d-1}\frac{\module{1 + t \kappa_{\delta,i}}}{\sqrt{1 + \kappa^2_{i, \delta}}} \diff \mathcal{H}^{d-1}(x,n) \diff t\\
    = & \int_0^R \int_{\Nor(A)} J_{d-1}(h_{\delta}) \prod_{i=1}^{d-1} \frac{\module{1 + t f_{\delta}(\kappa_i)}}{\sqrt{1 + f_{\delta}(\kappa_i)^2}} \diff \mathcal{H}^{d-1} (x,n) \diff t.
\end{align*}
Since $h_{\delta} - \Id$ is $\delta$-Lipschitz, we have \[ (1 - \delta)^{d-1} \leq J_{d-1}(h_{\delta}(x,n)) \leq (1 + \delta)^{d-1}. \] Hence, as $\delta$ tends to 0, $J_{d-1} h_{\delta}$ tends to $1$ and $f_{\delta}(\kappa)$ tends to $\kappa$ for all $\kappa \in \R$.
Lebesgue dominated convergence theorem gives:
\begin{align*}
\lim_{\delta \to 0} M_R(X^{\delta}) = & \int_{0}^R\int_{\Nor(A)} \prod_{i=1}^{d-1} \frac{\module{1 + t \kappa_i}}{\sqrt{1 + \kappa_i^2}} \diff \mathcal{H}^{d-1}(x,n) \diff t \\
 = & \; M_R(X). 
\end{align*}

For the second claim, assuming $X$ is subanalytic, we make use of the theory of normal cycles of subanalytic sets developed in \cite{FuSubAnalytic}. The function $d_X$ is a subanalytic aura of $X$ and as such we have convergence of the normal cycles $N_{X^{\delta}}$ of $X^{\delta}$ to that of $X$ for the flat norm as $\delta$ tends to 0. In particular, as intrinsic volumes are the normal cycles integrated against Lipschitz-Killing forms, we have for every $ 0 \leq i \leq d-1$, $\displaystyle \lim_{\delta \to 0} V_i(X^{\delta}) = V_i(X)$
and the last claim is a consequence of \Cref{eq:inegalite_MR} as $\delta$ tends to 0.
\end{proof}

Combining \Cref{thm:ConvIntrOffset} and \Cref{thm:persistent_rate} gives us a way to estimate $V_i(X)$ from $Y$:

\begin{corollary}[Linear rate of approximation for the intrinsic volumes]
\label{cor:cor_approx}
Let $X, Y$ be compact subsets of $\R^d$ and $\mu, \eps > 0$ be such that $d_H(X,Y) \leq \eps < \frac{1}{4}\reach_{\mu}(X)$.
If ${\reach(X) > 0}$, we have:
\begin{equation*}
    \module{V_i(X) - V^{\eps, R}_i(Y)} \leq \frac{\eps P(i,d)}{2\mu R^{d-i+1}} \left ( \Vol(X) + 2\Vol(X^{2\eps}) + M_R(X) + 2M_R(X^{2\eps}) \right ).
\end{equation*}
\end{corollary}

We prove \Cref{thm:milnor_2} using similar methods as before, except that the interleavings between persistence modules stem from the existence of a $(\eps, \delta)$-homotopy equivalence.

\begin{proof}[Proof of \Cref{thm:milnor_2}]
Let $x$ be any point of  $\R^d$ and let $M_a \coloneqq M \cap B(x,a)$ for any subset $M$ of $\R^d$ and $a \in \R$. Let $H_1 : [0, 1] \times Y \to Y$ be the
homotopy between $f \circ g$ and $\mathrm{Id}_Y$ coming from the definition of a $(\eps, \delta)$-homotopy equivalence. 
By assumption its restriction to $Y_a$ is a continuous map $H^a_1 : [0,1] \times Y_a \to Y_{a+ 2\delta}$. This yields a homotopy between $f \circ g : Y_a \to Y_{a+2\delta}$ and the inclusion $Y_a \xhookrightarrow{} Y_{a+2\delta}$. Letting $\delta'$ be a positive real such that $2\eps + 2\delta' \geq 2 \delta$, we obtain this commutative diagram of continuous maps:
\[\begin{tikzcd}
	& {X_{a+ \varepsilon}} & {X_{a +\varepsilon + \delta'}} \\
	{Y_a} &&& {Y_{a + 2 \varepsilon + \delta'}} & {Y_{a + 2 \varepsilon + 2\delta'}}
	\arrow[hook, from=1-2, to=1-3]
	\arrow["f", from=1-3, to=2-4]
	\arrow["\psi", shift left=2, from=1-3, to=2-5]
	\arrow["g", from=2-1, to=1-2]
	\arrow["\phi"', from=2-1, to=1-3]
	\arrow[hook, from=2-4, to=2-5]
\end{tikzcd}\]
Since the same holds symmetrically for the other homotopy $H^a_2$, we can apply the homology functor to obtain the following commutative diagram thanks to the homotopy between $\psi \circ \phi$ (resp. $\phi \circ \psi$) and the inclusion $Y_a \xhookrightarrow{} Y_{a+ 2\eps + 2\delta'}$ (resp. $X_a \xhookrightarrow{} X_{a+ 2\eps + 2\delta'}$):
\[\begin{tikzcd}
	{} & {H_*(X_a)} & {H_*\left (X_{a +\varepsilon + \delta'} \right)} & {H_*\left (X_{a +2\varepsilon + 2\delta'} \right)} & {} \\
	{} & {H_*(Y_a)} & {H_*\left (Y_{a +\varepsilon + \delta'} \right)} & {H_* \left(Y_{a +2\varepsilon + 2\delta'} \right)} & {}
	\arrow[dashed, from=1-1, to=1-2]
	\arrow[from=1-2, to=1-3]
	\arrow[from=1-2, to=2-3]
	\arrow[from=1-3, to=1-4]
	\arrow[from=1-3, to=2-4]
	\arrow[dashed, from=1-4, to=1-5]
	\arrow[dashed, from=2-1, to=2-2]
	\arrow[from=2-2, to=1-3]
	\arrow[from=2-2, to=2-3]
	\arrow[from=2-3, to=1-4]
	\arrow[from=2-3, to=2-4]
	\arrow[dashed, from=2-4, to=2-5]
\end{tikzcd}\]
Optimizing on $\delta'$, this yields $d_B(D^x_X, D^x_Y) \leq \max(\eps, \delta)$ for any $x \in \R^d$. Now following the same line of reasoning as the proof of \Cref{thm:ConvIntrOffset}, first bounding the $L^1$ norm of $Q_X - Q_Y$, and then retrieving its coefficients, we obtain the bound:
    \[\module{V_i(X) - V_i(Y)} \leq \frac{\max(\eps, \delta) P(i,d)}{2 R^{d-i+1}}\left ( \Vol(X) + M_R(X) + \Vol(Y) + M_R(Y) \right ). \]
\end{proof}

\begin{remark}[]The bounds obtained in \Cref{thm:persistent_rate}, \Cref{thm:ConvIntrOffset}, \Cref{cor:cor_approx} and \Cref{thm:milnor_2} depend on the parameter $R$ and explode when $R \to 0$ or $\infty$. While this is not critical for the asymptotic rate, it could be interesting to find ways of guessing good values of $R$ a priori. The main results stated in the introduction are obtained taking $R = 1$ and using \Cref{prop:bound}.
\end{remark}

\section{Computing persistent intrinsic volumes}
\label{sec:MonteCarlo}
We conclude this paper by discussing how one may compute our estimators in practice, assuming $Y$ is given as a finite set of points.
From the proof of \Cref{thm:persistent_rate} the $V^{\eps,R}_i(Y)$ are given by: 
\begin{equation*}
V^{\eps, R}_i(Y) =  \int_0^R \int_{\R^d} \chi(D^{\eps,x}_Y(r)) S_{i,d}(r)  \diff x \diff r,   
\end{equation*} 
where $S_{i,d}$ is a polynomial.

Computing this integral exactly would involve computing a $d$-dimensional family of persistence diagrams, which is computationally daunting. We may however easily approximate $V^{\eps,R}_i(Y)$ with arbitrary accuracy using a Monte-Carlo method by sampling points $x$ uniformly in the $R$-offset of $Y$ and computing the image persistence diagrams of their distance function $d_x$ using e.g., \cite{ImagePersistentComputation}. The number of trials required to reach accuracy $\delta$ is of the order of $\delta^{-2}$ times the variance of the random variable $\displaystyle V = \int_0^R \chi(D^{\eps,x}_Y(r)) S_{i,d}(r) \diff r$.

The variance of $V$ is bounded by $\sup \module{V}^2$ and we have: 
\begin{equation*}
\sup \module{V} \leq \sup_{x \in \R^d} N_0^R(D^{\eps,x}_Y)  \sup_{r \in [0,R]}  R S_{i,d}(r),
\end{equation*}
that is, a polynomial in $R$ times the maximum size of the image persistence diagram of $d_x$. Now by \Cref{prop:approx_diag}, the sizes of these image persistence diagrams are at most the sizes of the persistence diagrams of $(d_x)|_{X^{2\eps}}$.

One situation where we can uniformly bound the size of these diagrams is when $X$ is a semi-algebraic set, since by Thom-Milnor \cite{Milnor} the number of critical points of $(d_x)|_{X^{2\eps}}$ is bounded by a function of the degrees of the polynomial equalities and inequalities defining $X$. Another bound is provided by \Cref{lem:well-defined}. This bound is pessimistic in general as the smallest simplicial complex sandwiched between $Y^{\eps}$ and $Y^{3\eps}$ may have a large number of simplices, yet one expects the number of critical points of $(d_x)|_{C}$ to be much smaller on average. We leave more in-depth investigations of the computational aspects of our method for future work.

\section*{Acknowledgement}

We are grateful to the reviewers for their careful reading of the manuscript and for their constructive and insightful comments, which greatly helped us improve the paper.

\bibliographystyle{plain} 

\bibliography{ref.bib}

\end{document}